\documentclass[11pt]{amsart}

\usepackage[ansinew]{inputenc}
\usepackage[a4paper,dvips]{geometry}
\usepackage{latexsym}
\usepackage{amsfonts}
\usepackage{amsmath,amssymb}
\usepackage{bbm}
\usepackage[dvipsnames]{xcolor}
\usepackage{tikz}
\usetikzlibrary{decorations.pathreplacing,angles,quotes,calligraphy}
\usepackage{tikz-3dplot}
\usepackage{caption}
\usepackage{subcaption}

\usepackage{fixme}
\usepackage{graphicx}
\allowdisplaybreaks

\makeatletter
\newcommand{\addresseshere}{%
  \enddoc@text\let\enddoc@text\relax
}
\makeatother

\newtheorem{theorem}{Theorem}
\newtheorem{remark}[theorem]{Remark}
\newtheorem{lemma}[theorem]{Lemma}

\newtheorem{proposition}[theorem]{Proposition}
\newtheorem*{proposition*}{Proposition}
\newtheorem{corollary}[theorem]{Corollary}

\newcommand{\bv}[1]{\mathbf{#1}}

\newcommand{\Var}{{\mathbb V}\mathrm{ar}}

\newcommand{\cP}{{\mathcal P}}

\newcommand{\cL}{{\mathcal L}}

\title[The $\mathbb{E}\mathcal{L}_2-$discrepancy of jittered sampling]{\large On the expected $\mathcal{L}_2-$discrepancy of jittered sampling}

\author{Nathan Kirk}
\address{Queen's University Belfast, Belfast, United Kingdom.}
\email{nkirk09@qub.ac.uk}

\author{Florian Pausinger}
\address{Queen's University Belfast, Belfast, United Kingdom.}
\email{f.pausinger@qub.ac.uk}

\date{}

\begin{document}

%\subjclass[2010]{ }
\maketitle

%%%%%%%%%%%%%%%%%%%
%	Abstract
%%%%%%%%%%%%%%%%%%%

\begin{abstract}
For $m, d \in \mathbb{N}$, a jittered sample of $N=m^d$ points can be constructed by partitioning $[0,1]^d$ into $m^d$ axis-aligned equivolume boxes and placing one point independently and uniformly at random inside each box. We utilise a formula for the expected $\mathcal{L}_2-$discrepancy of stratified samples stemming from general equivolume partitions of $[0,1]^d$ which recently appeared, to derive a closed form expression for the expected $\mathcal{L}_2-$discrepancy of a jittered point set for any $m, d \in \mathbb{N}$. 
As a second main result we derive a similar formula for the expected Hickernell $\mathcal{L}_2-$discrepancy of a jittered point set which also takes all projections of the point set to lower dimensional faces of the unit cube into account.
%\\
%{\bf Last update: } 12/9/2022
\end{abstract}

%%%%%%%%%%%%%%%%%%%
%	Introduction
%%%%%%%%%%%%%%%%%%%

\section{Introduction}

\subsection{Discrepancy and jittered sampling}

In the classical setting, discrepancy theory concerns itself with the study of the irregularity of distribution of point sets contained in the unit cube. There are numerous notions and variants of discrepancy with wide ranging applications in numerical integration, computer graphics, machine learning and option pricing in financial mathematics to name just a few \cite{DP, leopill, lem, Nieder}. 

One such discrepancy measure is the so-called \textit{$\mathcal{L}_p-$discrepancy} for $1 \leq p < \infty$; in this paper we will focus our attention primarily on the case $p=2$ while briefly mentioning the star discrepancy in order to state the motivation for this note.
For a point set $\mathcal{P}_N = \{\mathbf{x}_1, \ldots, \mathbf{x}_N\}$ contained in $[0,1]^d$, we define the $\mathcal{L}_p-$discrepancy as
$$\mathcal{L}_p(\mathcal{P}_N) = \left( \int_{[0, 1]^d} \left| \frac{\# \left( \mathcal{P}_N \cap [0, \mathbf{x}) \right) }{N} - |[0, \mathbf{x})| \right|^p d\mathbf{x} \right)^{1/p}$$ 
in which $1 \leq p < \infty$, $\#(\mathcal{P}_N \cap [0, \mathbf{x}))$ counts the number of indices $1 \leq i \leq N$ such that $\mathbf{x}_i \in [0, \mathbf{x})$ and $|[0, \mathbf{x})|$ is the standard Lebesque measure of the subset $[0, \mathbf{x}) = \prod_{j=1}^d [0, x_j)$ with $\mathbf{x} = (x_1, \ldots, x_d)$. 
For an infinite sequence $\mathcal{X}$, we calculate the discrepancy of its first $N$ elements; i.e. the $\mathcal{L}_p-$discrepancy of $\mathcal{X}$ is defined to be the $\mathcal{L}_p-$discrepancy of the first $N$ terms of $\mathcal{X}$ for an arbitrary, but fixed $N$. 
Note that for ease of notation, we simply write $\mathcal{L}_p(\cdot)$ instead of $\mathcal{L}_{p,N}(\cdot)$ since we only work with finite point sets in this note.

As will be the case throughout this note, when $\mathcal{P}_N$ is a set of random samples the \textit{mean $p^{th}$ power $\mathcal{L}_p-$discrepancy} $\mathbb{E}\mathcal{L}^p_p(\mathcal{P}_N)$ is often utilised as the discrepancy measure, where $\mathbb{E}$ denotes the probabilistic expectation. 
Usually we simply refer to this as the mean (or expected) $\mathcal{L}_p-$discrepancy. 
The $\mathcal{L}_p-$discrepancy as defined above is simply taking the $\mathcal{L}_p$ norm of the discrepancy function (the deviation of the measure of a test set $[0, \mathbf{x})$ from the fraction of points lying inside $[0, \mathbf{x})$). 
One can also take the $\mathcal{L}_{\infty}$ norm of the discrepancy function to create another measure of irregularity of distribution called the \textit{star-discrepancy}, defined as $$D_N^*(\mathcal{P}) = \sup_{\mathbf{x} \in [0,1]^d} \left| \frac{\# \left( \mathcal{P}_N \cap [0, \mathbf{x}) \right) }{N} - |[0, \mathbf{x})| \right|.$$

The \textit{Hickernell $\mathcal{L}_2$-discrepancy} was introduced by Hickernell in \cite{Hickernell, Hickernell2}; note that some ideas can be traced back to \cite{zaremba}, while, to the best of our knowledge, the explicit definition appeared for the first time in \cite{Hickernell, Hickernell2}. It considers not only the ordinary $\mathcal{L}_2-$ discrepancy of a point set in the $d$-dimensional unit cube, but also the $\mathcal{L}_2-$discrepancy of all projections of the point set onto lower-dimensional faces of the $d$-dimensional unit cube. 

For any nonempty subset $s \subseteq \{1:d\}$ of coordinate indices, let $[0,1]^s$ denote the $|s|-$dimensional unit cube spanned by the coordinate axes in $s$ and let $\mathbb{N}_m^s$ be the set of $|s|-$dimensional vectors with coordinates in $s$ consisting only of entries from $\{1:m\}$. Likewise, if $\boldsymbol{\Omega} = \{\Omega_1, \ldots, \Omega_N\}$ is a jittered partition of $[0,1]^d$ and $\mathcal{P}_{\boldsymbol{\Omega}}$ is the jittered $N-$element point set obtained from the partition $\boldsymbol{\Omega}$, then $\boldsymbol{\Omega}^s = \{\Omega_1^s, \ldots, \Omega_N^s\}$ and $\mathcal{P}_{\boldsymbol{\Omega}}^s$ denote the projection of the partition and jittered point set onto $[0,1]^s$ respectively. Similarly, $\mathbf{x}^s$ denotes the projection of a general vector $\mathbf{x} \in [0,1]^d$ into $[0,1]^s$.
Then, the \textit{Hickernell $\mathcal{L}_2$-discrepancy} of a point set $\mathcal{P} \subset [0,1]^d$ is given by
\begin{equation*}
	D_{H,2}(\mathcal{P}) := \left( \sum_{\emptyset \neq s \subseteq \{1:d\}}  \int_{[0,1]^s} \left| \frac{\#(\mathcal{P}^s \cap [0,\mathbf{x}^s)}{N} - |[0, \mathbf{x}^s)| \right|^2 d\mathbf{x}^s \right)^{1/2}
\end{equation*}
or more concisely using the notation already established, 
\begin{equation}\label{eq:defzaremba}
	D_{H,2}(\mathcal{P}) := \left( \sum_{\emptyset \neq s \subseteq \{1:d\}} \mathcal{L}_2^2(\mathcal{P}^s) \right)^{1/2}
\end{equation}
where $\mathcal{P}^s$ denotes the projection of the sample $\mathcal{P}$ into the cube $[0,1]^s$.

In fact, Hickernell defined this discrepancy for general $p$: For $1 \leq p < \infty$, the \textit{Hickernell $\mathcal{L}_p-$discrepancy} is defined as
\begin{equation*}\label{eq:defhickernell}
	D_{H,p}(\mathcal{P}) := \left( \sum_{\emptyset \neq s \subseteq \{1:d\}} \mathcal{L}_p^p(\mathcal{P}^s) \right)^{1/p}.
\end{equation*}

Constructions of deterministic point sets are widely used in the context of numerical integration via quasi-Monte Carlo (QMC) methods due to the low discrepancy value leading to improved order of the integration error. We refer the reader to the Koksma-Hlawka inequality \cite{APTS, DP, Hickernell, Hickernell2, Hlawka, Koksma, PausSvane}. 
The optimal order for the $\mathcal{L}_2-$discrepancy of a finite point set contained inside $[0,1]^d$ for $d\geq 2$ is $\mathcal{O}\left( (\log N)^{\frac{d-1}{2}}/N \right)$, conversely for a set of i.i.d uniform random points (a Monte Carlo point set) the expected discrepancy has order $\mathcal{O}(N^{-1/2})$. 
An element of randomness is often desirable or even necessary in real world simulations. 
Therefore to achieve the best of both worlds, it is of interest to select a deterministic point set and utilise a randomisation technique. 
The output of this process is aptly called, a \textit{randomised quasi-Monte Carlo} (RQMC) point set; RQMC point sets often have the benefit of possessing better distribution properties than MC samples while also featuring an element of randomness, useful for simulation of real world phenomena. 

\textit{Classical jittered sampling} is an example of an RQMC point set.  For a fixed $m \in \mathbb{N}$, we form a jittered point set by partitioning $[0,1]^d$ into $m^d$ axis-aligned cubes of equal measure and placing a random point inside each of the $N=m^d$ cubes independently. See Figure \ref{fig:jitteredintro} for an illustration.  

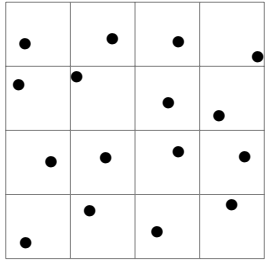
\begin{figure}
	\centering
	\begin{tikzpicture}[scale=0.85]
		
		\draw[step=1cm,gray,very thin] (8,0) grid (12,4);
		
		\node at (8.31,0.23) {$\bullet$}; 
		\node at (8.7,1.5) {$\bullet$}; 
		\node at (8.2,2.7) {$\bullet$}; 
		\node at (8.3,3.34) {$\bullet$};

		\node at (9.3,0.74) {$\bullet$}; 
		\node at (9.55,1.56) {$\bullet$}; 
		\node at (9.1,2.83) {$\bullet$}; 
		\node at (9.65,3.42) {$\bullet$};

		\node at (10.34,0.41) {$\bullet$}; 
		\node at (10.67,1.66) {$\bullet$}; 
		\node at (10.52,2.42) {$\bullet$}; 
		\node at (10.67,3.37) {$\bullet$};

		\node at (11.5,0.82) {$\bullet$}; 
		\node at (11.7,1.58) {$\bullet$}; 
		\node at (11.3,2.21) {$\bullet$}; 
		\node at (11.9,3.14) {$\bullet$};

	\end{tikzpicture}
	\caption{Jittered sampling for $d=2$ and $m=4$.}
	\label{fig:jitteredintro}
\end{figure}
		
\subsection{Exact formulas for the $\mathcal{L}_2$-discrepancy} \label{sec:L2}
Despite recurring criticism (see for example \cite{matousek}) the $\mathcal{L}_2$-discrepancy is a very popular notion of discrepancy mostly due to its simplicity. In contrast to other notions of discrepancy, we not only have an explicit formula for the $\mathcal{L}_2$-discrepancy of an arbitrary point set, known as Warnock's formula \cite{warnock}, as well as a fast implementation \cite{heinrich1, heinrich2}, but we also know the optimal order of the $\mathcal{L}_2$-discrepancy, i.e., we can look for point sets that are optimal with respect to $\mathcal{L}_2$-discrepancy; see \cite[Section 3.2]{DP} and references therein.

Warnock's formula as presented in \cite[Proposition 2.15]{DP} to calculate the $\cL_2$-discrepancy of a given point set holds for any point set $\cP=\{\bv x_0, \ldots, \bv {x}_{N-1} \} \in [0,1]^d$. We have
\begin{equation} \label{warnock}
\cL_2(\cP)^2 = \frac{1}{3^d} - \frac{2}{N} \sum_{n=0}^{N-1} \prod_{i=0}^d \frac{1-x_{n,i}^2}{2} + \frac{1}{N^2} \sum_{m,n=0}^{N-1} \prod_{i=0}^d \min(1-x_{m,i}, 1-x_{n,i}),
\end{equation}
in which $x_{n,i}$ is the $i$-th component of the $\bv {x}_n$. Apart from this general formula, it is often possible to give explicit, closed formulas for the $\mathcal{L}_2$-discrepancy of particular deterministic sequences. As an example we mention the exact formula for the symmetrized Hammersley point set as derived in \cite{kritz}.

\subsection{Main aim and outline}

B. Doerr \cite{Doerr} recently proved a tight star discrepancy estimate for jittered sampling. That is, the order of magnitude for the star discrepancy of a jittered point set is 
$$\mathbb{E} D_N^* \in \Theta \left( \frac{\sqrt{d} \sqrt{1+ \log(N/d)}}{N^{\frac{1}{2} + \frac{1}{2d}}} \right)$$
for all $m$ and $d$ with $m \geq d$. 
To the best of our knowledge, there is no such statement regarding an estimate or exact formula for the expected $\mathcal{L}_2-$discrepancy of a jittered sample for arbitrary dimension $d$ and number of points $N = m^d$. 
The results of Section \ref{sec:L2}, motivate us to look for such a closed formula. Our main result is as follows:

\begin{theorem}\label{thm:discforddimensions}
	Let $\boldsymbol{\Omega} =\{\Omega_{\boldsymbol{\mathit{i}}} : \boldsymbol{\mathit{i}} \in \mathbb{N}_m^d\}$ be a jittered partition of $[0,1]^d$ for $m \geq 2$. Then
	\begin{equation}\label{eq:L2discrepancy}
		\mathbb{E}\mathcal{L}_2^2(\mathcal{P}_{\boldsymbol{\Omega}}) = \frac{1}{m^{2d}} \left[ \left( \frac{m-1}{2}+\frac{1}{2} \right)^d - \left( \frac{m-1}{2}+\frac{1}{3} \right)^d \right].
	\end{equation}
\end{theorem}

\begin{remark}
In Proposition \ref{thm:projecteddiscrepancy} we derive a similar formula for the $\mathcal{L}_2$-discrepancy of projections of jittered point samples in $[0,1]^d$ to lower dimensional faces of the unit cube.
\end{remark}

As a second main result, we derive a closed formula for the Hickernell $\mathcal{L}_2$-discrepancy of jittered sampling:
\begin{theorem}\label{thm:hickernelldisc}
	Let $\boldsymbol{\Omega} = \{ \Omega_{\boldsymbol{i}} : \boldsymbol{i} \in \mathbb{N}_m^d\}$ be a jittered partition of $[0,1]^d$ for $m \geq 2$. Then $$\mathbb{E}D_{H,2}^2(\mathcal{P}_{\boldsymbol{\Omega}}) = \sum_{j=1}^{d} \frac{1}{m^{d+j}} \binom{d}{j} \left[ \left( \frac{m-1}{2} + \frac{1}{2} \right)^j - \left(  \frac{m-1}{2} + \frac{1}{3} \right)^j \right].$$
\end{theorem}

Based on Theorem \ref{thm:discforddimensions} we can derive the following corollary:

\begin{corollary}
	For a jittered partition $\boldsymbol{\Omega} =\{\Omega_{\boldsymbol{\mathit{i}}} : \boldsymbol{\mathit{i}} \in \mathbb{N}_m^d\}$ of $[0,1]^d$ containing $N=m^d$ points where $m \geq 2$, $$\mathbb{E}\mathcal{L}_2(\mathcal{P}_{\boldsymbol{\Omega}}) = \Theta \left( \frac{1}{m^{\frac{d}{2}-\frac{1}{2}}} \right) = \Theta \left( \frac{1}{N^{\frac{1}{2}+\frac{1}{2d}}} \right).$$
\end{corollary}

\begin{proof}
We can expand the difference in the brackets in Theorem \ref{thm:discforddimensions} using the binomial theorem:
\begin{align*}
\sum_{k=0}^d \binom{d}{k} \left(\frac{m-1}{2} \right)^k \left[ \left( \frac{1}{2} \right)^{d-k} - \left( \frac{1}{3} \right)^{d-k} \right]
\end{align*}
Each summand in this sum is positive. To get a lower bound, it is sufficient to truncate the sum and only consider the terms for $k=d$ and $k=d-1$, i.e., 
$$\mathbb{E}\mathcal{L}_2^2(\mathcal{P}_{\boldsymbol{\Omega}}) \geq  \frac{1}{m^{2d}}  \frac{d}{6}\left(\frac{m-1}{2} \right)^{d-1}.$$
Similarly, to obtain an upper bound, we simply notice that the leading terms cancel in the difference, such that we get a polynomial in $m$ of degree $d-1$.
Taking the square root and substituting $m=N^{1/d}$ leads to the assertion.
\end{proof}

Our proof utilizes a proposition from \cite{MKFP22} (see also \cite{MKFP21}) regarding the expected discrepancy of stratified samples obtained from an equivolume partition of the cube. 

\begin{proposition} [Proposition 3, \cite{MKFP22}]
	If $\boldsymbol{\Omega} = \{\Omega_1, \ldots, \Omega_N\}$ is an equivolume partition of a compact convex set $K \subset \mathbb{R}^d$ with $|K| > 0$, then 
	\begin{equation}\label{eq:originalequation}
		\mathbb{E}\mathcal{L}_2^2(\mathcal{P}_{\boldsymbol{\Omega}}) = \frac{1}{N^2 |K|} \sum_{i=1}^N \int_{K} q_i (\mathbf{x}) \left( 1 - q_i (\mathbf{x}) \right) d\mathbf{x}
	\end{equation} 
	with $q_i (\mathbf{x}) = \frac{\left| \Omega_i \cap [0, \mathbf{x}) \right|}{\left| \Omega_i \right|}.$
\end{proposition}

The structure of the paper is as follows. In Section 2, we derive the formula for the discrepancy of jittered sets in dimension 2. This calculation is much simpler than for arbitrary $d\geq 3$ and is used to highlight some difficulties in the general case. In Section 3 we introduce several lemmas needed to derive the proof of Theorem \ref{thm:discforddimensions} in Section 4. Section 5 contains the proof of Theorem \ref{thm:hickernelldisc}.

\section{The special case $d=2$}

For $m\geq 2$, let $\mathbb{N}_m^2$ denote the set of all ordered pairs with entries from $\{1, \ldots, m\}$ and let $\boldsymbol{\Omega} = \{\Omega_{(i, j)} : (i, j) \in \mathbb{N}_m^2\}$ be a jittered partition of $[0,1]^2$. We define two subsets of the unit square which will be of importance to us in the following derivation. 

For a given vector $(i,j) \in \mathbb{N}_m^2$, define $$I_v := \left\{ \mathbf{x} =(x_1,x_2) \in [0,1]^2 : \frac{i-1}{m} \leq x_1 \leq \frac{i}{m} \text{ and } \frac{j}{m} \leq x_2 \leq 1 \right\}$$ and similarly, $$I_h := \left\{ \mathbf{x} =(x_1,x_2) \in [0,1]^2 : \frac{i}{m} \leq x_1 \leq 1 \text{ and } \frac{j-1}{m} \leq x_2 \leq \frac{j}{m} \right\}.$$

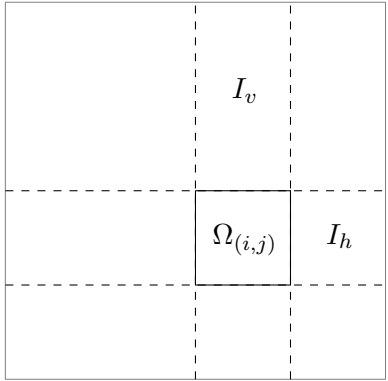
\begin{figure}
	
	\centering
	\begin{tikzpicture}[scale=1]
		\draw[step=5cm,gray,very thin] (0,0) grid (5,5);
		
		\draw[dashed] (2.5, 0) -- (2.5, 5);
		\draw[dashed] (3.75, 0) -- (3.75, 5);
		\draw[dashed] (0, 1.25) -- (5, 1.25);
		\draw[dashed] (0, 2.5) -- (5, 2.5);
		\draw (2.5, 2.5) -- (3.75, 2.5) -- (3.75, 1.25) -- (2.5, 1.25) -- (2.5, 2.5);
		
		\node at (3.15, 3.83) {$I_v$};
		\node at (4.4, 1.9) {$I_h$};
		\node at (3.15,1.85) {$\Omega_{(i, j)}$};
		
	\end{tikzpicture}
	
	\caption{The regions $I_h$ and $I_v$ for a given partition set $\Omega_{(i, j)}.$}
	\label{fig:regions}
\end{figure}

\begin{lemma}\label{lem:nonzerod2}
	Let $\boldsymbol{\Omega} =\{\Omega_{(i,j)} : (i, j) \in \mathbb{N}_m^2\}$ be a jittered partition of $[0,1]^2$ for $m \geq 2$ with $q_{(i,j)}(\mathbf{x}) = \frac{\left| \Omega_{(i,j)} \cap [0, \mathbf{x}) \right|}{\left| \Omega_{(i,j)} \right|}$. Then for each $(i,j) \in \mathbb{N}_m^2$,
	
	\begin{equation*}
		q_{(i,j)}(\mathbf{x}) \left( 1- q_{(i,j)}(\mathbf{x})\right) \neq 0 \text{ if and only if } \mathbf{x} \in \Omega_{(i,j)} \cup I_h \cup I_v
	\end{equation*}
\end{lemma}

\begin{proof}
	We prove the contrapositive statement. 
	If $\Omega_{(i,j)} \subset [0, \mathbf{x})$ then $\left| \Omega_{(i,j)} \cap [0, \mathbf{x}) \right| = |\Omega_{(i,j)}|$. 
	Hence, $q_{(i,j)} = 1$ and $q_{(i,j)}(\mathbf{x}) \left( 1- q_{(i,j)}(\mathbf{x})\right) = 0$. 
	Similarly, if $\Omega_{(i,j)} \cap [0, \mathbf{x}) = \emptyset$ then $\left| \Omega_{(i,j)} \cap [0, \mathbf{x}) \right| = 0$. 
	Hence, $q_{(i,j)} = 0$ and $q_{(i,j)}(\mathbf{x}) \left( 1- q_{(i,j)}(\mathbf{x})\right) = 0$. 
	Therefore, the integrand is non-zero if and only if $\Omega_{(i,j)}$ and $[0, \mathbf{x})$ have a nontrivial intersection and $\Omega_{(i,j)} \not\subseteq [0, \mathbf{x})$, i.e., when $\mathbf{x} \in \Omega_{(i,j)} \cup I_h \cup I_v$.
\end{proof}
We illustrate the regions for which $q_{(i,j)}(\mathbf{x}) \left( 1- q_{(i,j)}(\mathbf{x})\right) \neq 0$ in Figure \ref{fig:regions}.

\begin{lemma}\label{lem:qcases2d}
	For given $(i,j) \in \mathbb{N}_m^2$, we have the following
	\begin{equation} \label{eq:qcases2d}
		q_{(i,j)}(\mathbf{x}) = 
		\begin{cases}
			m^2 \left( x_1 - \frac{i-1}{m} \right) \left( x_2 - \frac{j-1}{m} \right), & \text{for } \mathbf{x} \in \Omega_{(i,j)} \\
			m \left( x_2 - \frac{j-1}{m} \right), & \text{for } \mathbf{x} \in I_h \\
			m \left( x_1 - \frac{i-1}{m} \right), & \text{for } \mathbf{x} \in I_v
		\end{cases}
	\end{equation}
\end{lemma}

\begin{proof}
	We discuss the case when $\mathbf{x} \in \Omega_{(i, j)}$ and note that the other cases can be derived in a similar manner. If $\mathbf{x} \in \Omega_{(i, j)}$, then by definition $\frac{i-1}{m} \leq x_1 \leq \frac{i}{m}$ and $\frac{j-1}{m} \leq x_2 \leq \frac{j}{m}$. Then recall that $$q_{(i, j)}(\mathbf{x}) = \frac{|\Omega_{(i, j)} \cap [0, \mathbf{x})|}{|\Omega_{(i, j)}|},$$ therefore 
	\begin{eqnarray}
		q_{(i, j)}(\mathbf{x}) = \frac{|\Omega_{(i, j)} \cap [0, \mathbf{x})|}{|\Omega_{(i, j)}|} 
		 = m^2 \left( x_1 - \frac{i-1}{m} \right) \left( x_2 - \frac{j-1}{m} \right) \nonumber
	\end{eqnarray}
	as required.
\end{proof}

\begin{theorem}
	Let $\boldsymbol{\Omega} = \{ \Omega_{\boldsymbol{\mathit{i}}} : \boldsymbol{\mathit{i}} \in \mathbb{N}_m^2\}$ be a jittered partition of $[0, 1]^2$ for $m \geq 2$, then 
	\begin{equation*}
		\mathbb{E}\mathcal{L}_2^2(\mathcal{P}_{\boldsymbol{\Omega}}) = \frac{6m-1}{36m^4}.
	\end{equation*}
\end{theorem}

\begin{proof}
To calculate the mean $\mathcal{L}_2-$discrepancy of two dimensional jittered sampling, we start from equation \eqref{eq:originalequation}. Let $K = [0,1]^2$, $N=m^2$ and assign a position vector $(i,j)$ with $1 \leq i, j \leq m$ to the set of the partition with lower left vertex $(i, j)$. Using this notation and Lemma \ref{lem:nonzerod2} we obtain 
	 \begin{equation}\label{eq:2dimensions}
	 	\mathbb{E}\mathcal{L}_2^2(\mathcal{P}_{\boldsymbol{\Omega}}) = \frac{1}{m^{2d}} \sum_{(i,j) \in \mathbb{N}_m^2} \int_{J} q_{(i,j)} (\mathbf{x}) \left( 1 - q_{(i,j)} (\mathbf{x}) \right) d\mathbf{x}
	\end{equation}
 	with $q_{(i,j)} (\mathbf{x}) = \frac{\left| \Omega_{(i,j)} \cap [0, \mathbf{x}) \right|}{\left| \Omega_{(i,j)} \right|}$ and $J = \Omega_{(i, j)} \cup I_h \cup I_v$. For a given set in position $(i,j)$ we can evaluate the integral:

	\begin{eqnarray}
		\int_{J} q_{(i,j)} (\mathbf{x}) \left( 1 - q_{(i,j)} (\mathbf{x}) \right) d\mathbf{x} &=& \int_{\frac{i-1}{m}}^{\frac{i}{m}} \int_{\frac{j-1}{m}}^{\frac{j}{m}} q_{(i,j)} (\mathbf{x}) \left( 1 - q_{(i,j)}(\mathbf{x}) \right) dx_2x_1 \nonumber \\ &+&  \int_{\frac{i}{m}}^{1} \int_{\frac{j-1}{m}}^{\frac{j}{m}} q_{(i,j)} (\mathbf{x}) \left( 1 - q_{(i,j)}(\mathbf{x})  \right) dx_2x_1 \nonumber \\ &+& \int_{\frac{i-1}{m}}^{\frac{i}{m}} \int_{\frac{j}{m}}^{1} q_{(i,j)} (\mathbf{x}) \left( 1 - q_{(i,j)} (\mathbf{x}) \right) dx_2x_1. \nonumber 
	\end{eqnarray}
	
Evaluating the integrals and using \eqref{eq:2dimensions} gives,
	
	\begin{eqnarray}
		\mathbb{E}\mathcal{L}_2^2(\mathcal{P}_{\boldsymbol{\Omega}}) &=& \frac{1}{m^4} \sum_{\boldsymbol{\mathit{i}} \in \mathbb{N}_m^2} \frac{5}{36m^2} + \frac{m-i}{6m^2} + \frac{m-j}{6m^2} \nonumber \\ &=& \frac{1}{m^4} \sum_{i, j=1}^m \frac{5}{36m^2} + \frac{m-i}{6m^2} + \frac{m-j}{6m^2} \nonumber \\ &=& \frac{6m-1}{36m^4}. \nonumber
	\end{eqnarray}	
	as required.
\end{proof}

\begin{corollary}
	For a jittered partition $\boldsymbol{\Omega}$ of $[0,1]^2$, $$\mathbb{E}\mathcal{L}_2(\mathcal{P}_{\boldsymbol{\Omega}}) = \Theta \left( N^{-3/4} \right).$$
\end{corollary}

\begin{proof}
This follows from $N=m^2$.
\end{proof}

\section{Several key lemmas}

For $m \geq 2$, let $\mathbb{N}_m^d$ denote the set of all $d-$tuples with entries from $\{1, \ldots, m\}$. Let $\boldsymbol{\Omega} = \{ \Omega_{\boldsymbol{\mathit{i}}} : \boldsymbol{\mathit{i}} = (i_1, \ldots, i_d) \in \mathbb{N}_m^d \}$ denote a jittered partition of $[0, 1]^d$ with $\boldsymbol{\mathit{i}} \in \mathbb{N}_m^d$ acting as the position vector of the jittered set $\Omega_{\boldsymbol{\mathit{i}}}$. If $u$ is a nonempty \emph{strict}, subset of $\{1:d\} = \{1, 2, \ldots, d\}$, then for every $\boldsymbol{\mathit{i}} \in \mathbb{N}_m^d$ we can define the set

\begin{equation}\label{eq:defofI}
	\mathcal{I}^u_{\boldsymbol{\mathit{i}}} = \left\{ \mathbf{x} \in [0,1]^d : \frac{i_j-1}{m} \leq x_j \leq \frac{i_j}{m} \text{ for } j \notin u, \frac{i_k}{m} \leq x_k \leq 1 \text{ for } k \in u \right\}.
\end{equation}

Note that there are $2^d-2$ such subsets for each $\boldsymbol{\mathit{i}} \in \mathbb{N}_m^d$. There are $d-1$ different types of subsets $\mathcal{I}^u_{\boldsymbol{\mathit{i}}}$ which can be classified by the cardinality of the set $u$, denoted throughout by the usual $|u|$. It is therefore natural to say that a set is of \textit{type}$-|u|$. See Figure \ref{fig:regions} for visual aid and examples of the regions $\mathcal{I}^u_{\boldsymbol{\mathit{i}}}$ for $d=2$. Specifically, $I_v$ can be identified with $\mathcal{I}^{\{2\}}_{\boldsymbol{\mathit{i}}}$ and similarly $I_h$ with $\mathcal{I}^{\{1\}}_{\boldsymbol{\mathit{i}}}$ in the two dimensional case.

\begin{lemma}\label{lem:regionsddimensions}
	Let $\boldsymbol{\Omega} =\{\Omega_{\boldsymbol{\mathit{i}}} : \boldsymbol{\mathit{i}} \in \mathbb{N}_m^d\}$ be a jittered partition of $[0,1]^d$ for $m \geq 2$ with $q_{\boldsymbol{\mathit{i}}}(\mathbf{x}) = \frac{\left| \Omega_{\boldsymbol{\mathit{i}}} \cap [0, \mathbf{x}) \right|}{\left| \Omega_{\boldsymbol{\mathit{i}}} \right|}$. Then for each $\boldsymbol{\mathit{i}} \in \mathbb{N}_m^d$,
	\begin{equation*}
		q_{\boldsymbol{\mathit{i}}}(\mathbf{x}) \left( 1- q_{\boldsymbol{\mathit{i}}}(\mathbf{x})\right) \neq 0 \text{ if and only if } \mathbf{x} \in \Omega_{\boldsymbol{\mathit{i}}} \cup \bigcup \mathcal{I}^u_{\boldsymbol{\mathit{i}}}
	\end{equation*}
in which the union is over all subsets $\mathcal{I}^u_{\boldsymbol{\mathit{i}}}$ as defined in \eqref{eq:defofI} for $\emptyset \neq u \subset \{1:d\}$.
\end{lemma}

\begin{proof}
	It is more convenient to prove this via the contrapositive statement. That is, 
	\begin{equation*}
		q_{\boldsymbol{\mathit{i}}}(\mathbf{x}) \left( 1- q_{\boldsymbol{\mathit{i}}}(\mathbf{x})\right) = 0 \text{ if and only if } \mathbf{x} \notin \Omega_{\boldsymbol{\mathit{i}}} \cup \bigcup \mathcal{I}^u_{\boldsymbol{\mathit{i}}}.
	\end{equation*}
	
Hence, given $\boldsymbol{\mathit{i}} \in \mathbb{N}_m^d$ suppose that $\mathbf{x} \notin \Omega_{\boldsymbol{\mathit{i}}} \cup \bigcup \mathcal{I}^u_{\boldsymbol{\mathit{i}}}$ for any nonempty choice of $u \subset \{1:d\}$. 
Then by definition, we are either in the situation that: there exists at least one $j \in \{1:d\}$ such that $x_j \leq \frac{i_j - 1}{m}$ which implies $\Omega_{\boldsymbol{\mathit{i}}} \cap [0, \mathbf{x}) = \emptyset$. Hence, $$q_{\boldsymbol{\mathit{i}}}(\mathbf{x}) = \frac{|\Omega_{\boldsymbol{\mathit{i}}} \cap [0, \mathbf{x})|}{|\Omega_{\boldsymbol{\mathit{i}}}|} = 0.$$ Or we have that $x_j \geq \frac{i_j}{m}$ for all $j \in \{1:d\}$ which implies $\Omega_{\boldsymbol{\mathit{i}}} \subseteq [0, \mathbf{x})$. 
In this case $$q_{\boldsymbol{\mathit{i}}}(\mathbf{x}) = \frac{|\Omega_{\boldsymbol{\mathit{i}}} \cap [0, \mathbf{x})|}{|\Omega_{\boldsymbol{\mathit{i}}}|} = \frac{ |\Omega_{\boldsymbol{\mathit{i}}}|}{|\Omega_{\boldsymbol{\mathit{i}}}|} = 1.$$
	
In both cases, $q_{\boldsymbol{\mathit{i}}}(\mathbf{x}) \left( 1- q_{\boldsymbol{\mathit{i}}}(\mathbf{x})\right) = 0$ as required. 
In the other direction, suppose, $q_{\boldsymbol{\mathit{i}}}(\mathbf{x}) \left( 1- q_{\boldsymbol{\mathit{i}}}(\mathbf{x})\right) = 0$.
Then either $q_{\boldsymbol{\mathit{i}}}(\mathbf{x}) = 0 \Rightarrow |\Omega_{\boldsymbol{\mathit{i}}} \cap [0, \mathbf{x})| = 0 \Rightarrow \Omega_{\boldsymbol{\mathit{i}}} \cap [0, \mathbf{x}) = \emptyset$, and therefore $\mathbf{x} \notin \Omega_{\boldsymbol{\mathit{i}}} \cup \bigcup \mathcal{I}_{\boldsymbol{\mathit{i}}}^u$. 
Alternatively if $q_{\boldsymbol{\mathit{i}}}(\mathbf{x}) = 1$, then $|\Omega_{\boldsymbol{\mathit{i}}} \cap [0, \mathbf{x})| = |\Omega_{\boldsymbol{\mathit{i}}}| \Rightarrow \Omega_{\boldsymbol{\mathit{i}}} \subset [0, \mathbf{x}) \Rightarrow \mathbf{x} \notin \Omega_{\boldsymbol{\mathit{i}}} \cup \mathcal{I}_{\boldsymbol{\mathit{i}}}^u$ for any nonempty $u \subset \{1:d\}$ since $\frac{i_j}{m} \leq x_j \leq 1$ for all $j \in \{1:d\}$. 
The statement of the Lemma is now proved.
\end{proof}

We can now adapt equation \eqref{eq:originalequation}.
Let $K = [0,1]^d$ and $N=m^d$, then
	\begin{equation} \label{eq:bequationddimension}
		\mathbb{E}\mathcal{L}_2^2(\mathcal{P}_{\boldsymbol{\Omega}}) = \frac{1}{(m^d)^2} \sum_{\boldsymbol{\mathit{i}} \in \mathbb{N}_m^d} \int_{[0,1]^d} q_{\boldsymbol{\mathit{i}}} (\mathbf{x}) \left( 1 - q_{\boldsymbol{\mathit{i}}}(\mathbf{x}) \right) d\mathbf{x}
	\end{equation}
	where $q_{\boldsymbol{\mathit{i}}}(\mathbf{x}) = \frac{\left| \Omega_{\boldsymbol{\mathit{i}}} \cap [0, \mathbf{x}) \right|}{\left| \Omega_{\boldsymbol{\mathit{i}}} \right|}.$ 
Next, we discard those regions of $[0,1]^d$ where the integrand vanishes as shown in Lemma \ref{lem:regionsddimensions}. 
Therefore, we get, 
\begin{equation}\label{eq:simplifiedddimensions}
\mathbb{E}\mathcal{L}_2^2(\mathcal{P}_{\boldsymbol{\Omega}}) = \frac{1}{m^{2d}} \sum_{\boldsymbol{\mathit{i}} \in \mathbb{N}_m^d} \left[ \int_{\Omega_{\boldsymbol{\mathit{i}}}} q_{\boldsymbol{\mathit{i}}} (\mathbf{x}) \left( 1 - q_{\boldsymbol{\mathit{i}}}(\mathbf{x}) \right) d\mathbf{x} + \sum_{ \emptyset \neq u \subset \{1:d\}} \int_{\mathcal{I}^u_{\boldsymbol{\mathit{i}}}} q_{\boldsymbol{\mathit{i}}} (\mathbf{x}) \left( 1 - q_{\boldsymbol{\mathit{i}}}(\mathbf{x}) \right) d\mathbf{x} \right].
\end{equation}
%\end{proof}

\begin{lemma}\label{lem:qcasesddimensions}
	For a given $\boldsymbol{\mathit{i}} \in \mathbb{N}_m^d$
	\begin{equation} \label{eq:qcasesddimensions}
		q_{\boldsymbol{\mathit{i}}}(\mathbf{x}) = 
		\begin{cases}
			m^d \prod_{j=1}^{d} \left( x_j - \frac{i_j-1}{m} \right), & \text{for } \mathbf{x} \in \Omega_{\boldsymbol{\mathit{i}}} \\
			
			m^{d-|u|} \prod_{j \notin u} \left( x_j - \frac{i_j-1}{m} \right), & \text{for } \mathbf{x} \in \mathcal{I}_{\boldsymbol{\mathit{i}}}^u.
		\end{cases}
	\end{equation}
\end{lemma}

\begin{proof}
	Fix $\boldsymbol{\mathit{i}} = (i_1, \ldots, i_d) \in \mathbb{N}_m^d$. Both cases can be proven simultaneously by noting that $\mathcal{I}_{\boldsymbol{\mathit{i}}}^{\emptyset} = \Omega_{\boldsymbol{\mathit{i}}}$. Suppose that $\mathbf{x} \in \mathcal{I}_{\boldsymbol{\mathit{i}}}^u$ for any $u \subset \{1:d\}$, then by definition $$\frac{i_j-1}{m} \leq x_j \leq \frac{i_j}{m} \text{ for } j \notin u \textnormal{  and } \frac{i_k}{m} \leq x_k \leq 1 \text{ for } k \in u.$$ Now to determine the desired expression for the function $q_{\boldsymbol{\mathit{i}}}(\mathbf{x})$, we are interested in the intersection between the jittered set $\Omega_{\boldsymbol{\mathit{i}}}$ and the test box $[0, \mathbf{x})$. From the observation above, we can conclude that 
	$$\left| \Omega_{\boldsymbol{\mathit{i}}} \cap [0, \mathbf{x}) \right| = \frac{1}{m^{|u|}} \prod_{j \notin u} \left( x_j - \frac{i_j - 1}{m} \right).$$ 
Hence, 
\begin{eqnarray}
	q_{\boldsymbol{\mathit{i}}}(\mathbf{x}) = \frac{\left| \Omega_{\boldsymbol{\mathit{i}}} \cap [0, \mathbf{x}) \right|}{\left| \Omega_{\boldsymbol{\mathit{i}}} \right|}  = \frac{ \frac{1}{m^{|u|}} \prod_{j \notin u} \left( x_j - \frac{i_j - 1}{m} \right) }{1/m^{d}}  = m^{d-|u|} \prod_{j \notin u} \left( x_j - \frac{i_j - 1}{m} \right) \nonumber
\end{eqnarray}
	
	\noindent
	as required. For completeness, when $u=\emptyset$ (i.e. $I_{\boldsymbol{\mathit{i}}}^\emptyset = \Omega_{\boldsymbol{\mathit{i}}}$) the expression becomes $$q_{\boldsymbol{\mathit{i}}}(\mathbf{x}) = m^d \prod_{j=1}^{d} \left( x_j - \frac{i_j-1}{m} \right)$$ as required.
\end{proof}

\begin{lemma}\label{lem:integraloverI}
Let $\boldsymbol{\mathit{i}} \in \mathbb{N}_m^d$ and let $u$ be a nonempty, strict subset of $\{1:d\}$. Then
	\begin{equation*}
		\int_{I^u_{\boldsymbol{\mathit{i}}}} q_{\boldsymbol{\mathit{i}}}(\mathbf{x}) \left( 1 - q_{\boldsymbol{\mathit{i}}}(\mathbf{x}) \right) d\mathbf{x} = \left( \frac{3^{d-|u|} - 2^{d-|u|}}{(6m)^{d-|u|}} \right) \prod_{j \in u} \left( 1-\frac{i_j}{m} \right).
	\end{equation*}

\noindent
Moreover,

	\begin{equation*}
			\int_{\Omega_{\boldsymbol{\mathit{i}}}} q_{\boldsymbol{\mathit{i}}}(\mathbf{x}) \left( 1 - q_{\boldsymbol{\mathit{i}}}(\mathbf{x}) \right) d\mathbf{x} =  \frac{3^{d} - 2^{d}}{(6m)^{d}}.
	\end{equation*}

\end{lemma}

\begin{proof}
For a given $\boldsymbol{\mathit{i}} \in \mathbb{N}_m^d$, consider a type$-|u|$ subset of $[0,1]^d$. We can suppose that $u = \{1, 2, \ldots, |u|\} \subset \{1:d\}$ without loss of generality. Then using the general form of $q_{\boldsymbol{\mathit{i}}}(\mathbf{x})$, the second integral in \eqref{eq:simplifiedddimensions} has the form
	
	\begin{align}
		\int_{I^u_{\boldsymbol{\mathit{i}}}} q_{\boldsymbol{\mathit{i}}}(\mathbf{x}) (1 - q_{\boldsymbol{\mathit{i}}}(\mathbf{x})) d\mathbf{x} &= \int_{I^u_{\boldsymbol{\mathit{i}}}} q_{\boldsymbol{\mathit{i}}}(\mathbf{x}) - q^2_{\boldsymbol{\mathit{i}}}(\mathbf{x}) d\mathbf{x} \nonumber \\&= m^{d-|u|} \int_{I^u_{\boldsymbol{\mathit{i}}}} \prod_{j \notin u} \left( x_j - \frac{i_j-1}{m} \right) d\mathbf{x} - m^{2(d-|u|)} \int_{I^u_{\boldsymbol{\mathit{i}}}} \prod_{j \notin u} \left( x_j - \frac{i_j-1}{m} \right)^2 d\mathbf{x} \nonumber \\ &= m^{d-|u|} \int_{\frac{i_1}{m}}^1 \cdots \int_{\frac{i_{|u|}}{m}}^1 \prod_{j=|u|+1}^{d} \int_{\frac{i_j - 1}{m}}^{\frac{i_j}{m}} \left( x_j - \frac{i_j-1}{m} \right) dx_j dx_{|u|} \ldots x_1 \nonumber \\ & \hspace{1.3cm} - m^{2(d-|u|)} \int_{\frac{i_1}{m}}^1 \cdots \int_{\frac{i_{|u|}}{m}}^1 \prod_{j=|u|+1}^{d} \int_{\frac{i_j - 1}{m}}^{\frac{i_j}{m}} \left( x_j - \frac{i_j-1}{m} \right)^2 dx_j dx_{|u|} \ldots x_1 \nonumber \\ &= m^{d-|u|} \left( \frac{1}{2m^2} \right)^{d-|u|} \left( 1 - \frac{i_{|u|}}{m} \right) \cdots \left( 1 - \frac{i_1}{m} \right) \nonumber \\ & \hspace{1.3cm} - m^{2(d-|u|)} \left( \frac{1}{3m^3} \right)^{d-|u|} \left( 1 - \frac{i_{|u|}}{m} \right) \cdots \left( 1 - \frac{i_1}{m} \right) \nonumber \\ &= \frac{1}{(2m)^{d-|u|}} \prod_{j \in u} \left(1 - \frac{i_j}{m} \right) - \frac{1}{(3m)^{d-|u|}} \prod_{j \in u} \left(1 - \frac{i_j}{m} \right) \nonumber \\ &= \left( \frac{1}{(2m)^{d-|u|}} - \frac{1}{(3m)^{d-|u|}} \right) \prod_{j \in u} \left(1 - \frac{i_j}{m} \right) \nonumber \\ &= \left( \frac{3^{d-|u|} - 2^{d-|u|} }{(6m)^{d-|u|}} \right) \prod_{j \in u} \left(1 - \frac{i_j}{m} \right) \nonumber
	\end{align}
	
	\noindent
as required. 
To show the second statement in the lemma we use the formula of Lemma \ref{lem:qcasesddimensions} for $q_{\boldsymbol{\mathit{i}}}(\mathbf{x})$ and $\mathbf{x} \in \Omega_{\boldsymbol{\mathit{i}}}$. We notice that the calculation is similar to the above resulting in
\begin{equation*}
		\int_{\Omega_{\boldsymbol{\mathit{i}}}} q_{\boldsymbol{\mathit{i}}}(\mathbf{x}) (1 - q_{\boldsymbol{\mathit{i}}}(\mathbf{x})) d\mathbf{x} = \frac{3^{d} - 2^{d} }{(6m)^{d}}.
\end{equation*}
	
\end{proof}

\section{Proof of Theorem \ref{thm:discforddimensions} }

%\begin{proof}
	Starting from \eqref{eq:simplifiedddimensions}, we use Lemma \ref{lem:integraloverI} to rewrite 
	$$	\mathbb{E}\mathcal{L}_2^2(\mathcal{P}_{\boldsymbol{\Omega}}) = \frac{1}{m^{2d}} \sum_{\boldsymbol{\mathit{i}} \in \mathbb{N}_m^d} \left[ \int_{\Omega_{\boldsymbol{\mathit{i}}}} q_{\boldsymbol{\mathit{i}}} (\mathbf{x}) \left( 1 - q_{\boldsymbol{\mathit{i}}}(\mathbf{x}) \right) d\mathbf{x} + \sum_{ \emptyset \neq u \subset \{1:d\}} \int_{\mathcal{I}^u_{\boldsymbol{\mathit{i}}}} q_{\boldsymbol{\mathit{i}}} (\mathbf{x}) \left( 1 - q_{\boldsymbol{\mathit{i}}}(\mathbf{x}) \right) d\mathbf{x} \right]$$
	as
	$$\mathbb{E}\mathcal{L}_2^2(\mathcal{P}_{\boldsymbol{\Omega}}) = \frac{1}{m^{2d}} \sum_{\boldsymbol{\mathit{i}} \in \mathbb{N}_m^d} \left[ \frac{3^d - 2^d}{(6m)^d} + \sum_{ \emptyset \neq u \subset \{1:d\}} \left( \frac{3^{d-|u|} - 2^{d-|u|} }{(6m)^{d-|u|}} \right) \prod_{j \in u} \left( 1 - \frac{i_j}{m} \right) \right].$$
	
	We gather the terms in the last line by including the case $u=\emptyset$ in the summation over subsets of $\{1:d\}$.
	
	\begin{eqnarray}\label{eq:doublesum}	\mathbb{E}\mathcal{L}_2^2(\mathcal{P}_{\boldsymbol{\Omega}})&=& \frac{1}{m^{2d}} \sum_{\boldsymbol{\mathit{i}} \in \mathbb{N}_m^d} \left[ \sum_{u \subset \{1:d\}} \left( \frac{3^{d-|u|} - 2^{d-|u|} }{(6m)^{d-|u|}} \right) \prod_{j \in u} \left( 1 - \frac{i_j}{m} \right) \right] \nonumber \\ &=& \frac{1}{m^{2d}} \sum_{\boldsymbol{\mathit{i}} \in \mathbb{N}_m^d} \sum_{k=0}^d \sum_{\substack{u \subset \{1:d\} \\ |u|=k}} \frac{3^{d-k} - 2^{d-k} }{(6m)^{d-k}} \prod_{j \in u} \left(1-\frac{i_j}{m} \right) \nonumber \\ &=& \frac{1}{m^{2d}} \sum_{k=0}^d \frac{3^{d-k} - 2^{d-k} }{(6m)^{d-k}} \sum_{\substack{u \subset \{1:d\} \\ |u|=k}} \underbrace{\sum_{\boldsymbol{\mathit{i}} \in \mathbb{N}_m^d} \prod_{j \in u} \left( 1 - \frac{i_j}{m} \right) }_{(\ast)}  
\end{eqnarray}
	
	The quantity $(\ast)$ is equal to the same value for all subsets $u$ with $|u|=k>0$. Hence we calculate the double sum in \eqref{eq:doublesum} by letting $u=\{1, 2, \ldots, k\}$, setting $j_l = m - i_l$ for $1 \leq l \leq k$ and noting that there are $\binom{d}{k}$ subsets of $\{1:d\}$ of size $k$. Thus, 
	
	\begin{eqnarray}
		\sum_{\substack{u \subset \{1:d\} \\ |u|=k}} \sum_{\boldsymbol{\mathit{i}} \in \mathbb{N}_m^d} \prod_{j \in u} \left( 1 - \frac{i_j}{m} \right) &=&\sum_{i_{k+1}, \ldots, i_d =1}^m  \binom{d}{k} \sum_{i_1, \ldots, i_k =1}^m \prod_{j=1}^k \left(1 - \frac{i_j}{m} \right) \nonumber \\ &=& m^{d-k} \binom{d}{k} \sum_{j_1, \ldots, j_k =0}^{m-1} \left( \frac{j_1 j_2 \cdots j_k}{m^k} \right) \nonumber \\ &=& m^{d-k} \binom{d}{k} \cdot m^{-k} \left( \frac{m(m-1)}{2} \right)^k \nonumber \\ &=& m^{d-k} \binom{d}{k} \left( \frac{m-1}{2} \right)^k \nonumber
	\end{eqnarray}
	
	Incorporating the last derivation into \eqref{eq:doublesum}, we obtain
	
	\begin{eqnarray}
		\mathbb{E}\mathcal{L}_2^2(\mathcal{P}_{\boldsymbol{\Omega}}) &=& \frac{1}{m^{2d}}  \sum_{k=0}^d \binom{d}{k} \left[ \left( \frac{1}{2} \right)^{d-k} - \left( \frac{1}{3} \right)^{d-k} \right] \left(\frac{m-1}{2} \right)^k \nonumber \\ &=& \frac{1}{m^{2d}} \left[ \left( \frac{m-1}{2}+\frac{1}{2} \right)^d - \left( \frac{m-1}{2}+\frac{1}{3} \right)^d \right] \nonumber
	\end{eqnarray}
	 \noindent
	as required.
	
%\end{proof}

\section{Proof of Theorem \ref{thm:hickernelldisc}}
For the Hickernell discrepancy we are required to calculate the discrepancy of all projections of the point set in addition to that of the original set. Therefore, we first derive a formula for the expected $\mathcal{L}_2-$discrepancy of a $d-$dimensional jittered point set projected onto a lower dimensional face of the unit cube $[0,1]^d$.

\begin{proposition}\label{thm:projecteddiscrepancy}
	Let $\boldsymbol{\Omega} = \{\Omega_{\boldsymbol{i}} : \boldsymbol{i} \in \mathbb{N}_m^d \}$ be a jittered partition of $[0,1]^d$ for $m \geq 2$. For a nonempty subset $s \subseteq \{1:d\}$, we have
		\begin{equation}
		\mathbb{E}\mathcal{L}_2^2(\mathcal{P}_{\boldsymbol{\Omega}}^s) = \frac{1}{m^{d+|s|}} \left[ \left( \frac{m-1}{2} + \frac{1}{2} \right)^{|s|} - \left(  \frac{m-1}{2} + \frac{1}{3} \right)^{|s|} \right].
		\end{equation}
\end{proposition}

\begin{proof}
	Suppose $\mathcal{P}_{\boldsymbol{\Omega}}$ is a jittered sample contained in $[0,1]^d$ with $N=m^d$ points. We follow the notation and general method as presented in \cite{MKFP21} and consider the random variable $$Z_{\mathbf{x}^s}(\mathcal{P}_{\boldsymbol{\Omega}}^s) = \frac{\#(\mathcal{P}^s_{\boldsymbol{\Omega}} \cap [0, \mathbf{x}^s))}{N}$$ where $\mathcal{P}^s_{\boldsymbol{\Omega}}$ denotes the projected jittered point set into the unit cube $[0,1]^s$. It is easy to observe that for some nonempty $s \subseteq \{1:d\}$, the projection of a $d-$dimensional jittered partition onto $[0,1]^s$ is again a jittered partition (but with more than one point in each set) and hence equivolume. By Proposition 1 from \cite{MKFP21}, we can conclude that 
	\begin{equation}\label{eq:unbiased}
	\mathbb{E}(Z_{\mathbf{x}^s}(\mathcal{P}_{\boldsymbol{\Omega}}^s)) = |[0, \mathbf{x}^s)|.
	\end{equation}
	From this unbiasedness of the variable $Z_{\mathbf{x}^s} = Z_{\mathbf{x}^s}(\mathcal{P}_{\boldsymbol{\Omega}}^s)$ and applying Tonelli's theorem with $p=2$, we have  
	\begin{eqnarray} \label{eq:variance}
	\mathbb{E}\mathcal{L}_2^2(\mathcal{P}_{\boldsymbol{\Omega}}^s) &=& \int_{[0,1]^s} \mathbb{E} \left| Z_{\mathbf{x}^s} - |[0, \mathbf{x}^s)| \right|^2 d\mathbf{x}^s \nonumber \\ &=& \int_{[0,1]^s} \mathbb{E} \left|Z_{\mathbf{x}^s} - \mathbb{E}Z_{\mathbf{x}^s} \right|^2 d\mathbf{x}^s \nonumber \\ &=& \int_{[0,1]^s} \Var(Z_{\mathbf{x}^s}) d\mathbf{x}^s
	\end{eqnarray}
	with the last equality due to the expression $\mathbb{E}|Z_{\mathbf{x}^s} - \mathbb{E}Z_{\mathbf{x}^s}|^2$ which denotes the second central moment of $Z_{\mathbf{x}^s}$, i.e. the variance.
	
	Now we take note that for a given projection of $\mathcal{P}_{\boldsymbol{\Omega}}$ into dimension $d-1$, the number of jittered sets is reduced by a factor of $m$ with each set containing $m$ points. Continuing in a similar manner, when projecting onto a dimension $1 \leq |s| < d$, there are $\frac{N}{m^{d-|s|}}$ jittered sets each containing $m^{d-|s|}$ points. Therefore, the random variable $N \cdot Z_{\mathbf{x}^s} = \#(\mathcal{P}_{\boldsymbol{\Omega}}^s \cap [0, \mathbf{x}^s) )$ can be treated as the sum of $\frac{N}{m^{d-|s|}}$ groups of $m^{d-|s|}$ identical Bernoulli variables with success probabilities $\tilde{q}_{\boldsymbol{i}}(\mathbf{x}^s)$ for $\boldsymbol{i} \in \mathbb{N}_m^s$ where $$\tilde{q}_{\boldsymbol{i}}(\mathbf{x}^s) = \frac{|\Omega_{\boldsymbol{i}}^s \cap [0, \mathbf{x}^s)|}{|\Omega_{\boldsymbol{i}}^s|}.$$
	
	\noindent
	Therefore $$N^2 \Var(Z_{\mathbf{x}^s}(\mathcal{P}_{\boldsymbol{\Omega}}^s)) = \Var(N \cdot Z_{\mathbf{x}^s}(\mathcal{P}_{\boldsymbol{\Omega}}^s)) = \sum_{\boldsymbol{i} \in \mathbb{N}_m^s} m^{d-|s|} \tilde{q}_{\boldsymbol{i}}(\mathbf{x}^s) (1 - \tilde{q}_{\boldsymbol{i}}(\mathbf{x}^s)).$$
	Entering the last into \eqref{eq:variance} with $N=m^d$ yields
	\begin{equation}\label{eq:projectedintegral}
	\mathbb{E}\mathcal{L}_2^2(\mathcal{P}_{\boldsymbol{\Omega}}^s) = \frac{1}{m^{d+|s|}} \sum_{\boldsymbol{i} \in \mathbb{N}_m^s} \int_{[0,1]^s} \tilde{q}_{\boldsymbol{i}}(\mathbf{x}^s) (1 - \tilde{q}_{\boldsymbol{i}}(\mathbf{x}^s)) d\mathbf{x}^s.
	\end{equation}
	
	For a given projection onto $[0,1]^s$ with $\emptyset \neq s \subseteq \{1:d\}$, the probabilities $\tilde{q}_{\boldsymbol{i}}(\mathbf{x}^s)$ have the same form as stated in Lemma 9 in the new appropriate dimension $|s|$. Hence \eqref{eq:projectedintegral} can be simplified in a similar manner to the standard $\mathcal{L}_2-$discrepancy calculation and can be written as in the statement of the proposition, i.e. 
	
	\begin{equation}
		\mathbb{E}\mathcal{L}_2^2(\mathcal{P}_{\boldsymbol{\Omega}}^s) = \frac{1}{m^{d+|s|}} \left[ \left( \frac{m-1}{2} + \frac{1}{2} \right)^{|s|} - \left(  \frac{m-1}{2} + \frac{1}{3} \right)^{|s|} \right].
	\end{equation}
\end{proof}

Proposition \ref{thm:projecteddiscrepancy} is our main tool in the proof of Theorem \ref{thm:hickernelldisc}.

\begin{proof}[Proof of Theorem \ref{thm:hickernelldisc}]
	Let $\mathcal{P}_{\boldsymbol{\Omega}}$ be a jittered point set contained in $[0,1]^d$. We have that
	\begin{equation*}
		\mathbb{E}D_{H,2}^2(\mathcal{P}_{\boldsymbol{\Omega}}) := \sum_{\emptyset \neq s \subseteq \{1:d\}} \mathbb{E}\mathcal{L}_2^2(\mathcal{P}^s_{\boldsymbol{\Omega}}).
	\end{equation*}
	From Proposition \ref{thm:projecteddiscrepancy} we obtain,
	\begin{eqnarray}
		\mathbb{E}D_{H,2}^2(\mathcal{P}_{\boldsymbol{\Omega}}) &=& \sum_{\emptyset \neq s \subseteq \{1:d\}} \frac{1}{m^{d+|s|}} \left[ \left( \frac{m-1}{2} + \frac{1}{2} \right)^{|s|} - \left(  \frac{m-1}{2} + \frac{1}{3} \right)^{|s|} \right] \nonumber \\ &=& \sum_{j=1}^d \frac{1}{m^{d+j}} \binom{d}{j} \left[ \left( \frac{m-1}{2} + \frac{1}{2} \right)^j - \left(  \frac{m-1}{2} + \frac{1}{3} \right)^j \right]. \nonumber
	\end{eqnarray}
	where the last equality is due to the fact that we have the uniform distribution inside each original jittered set, hence the projection onto any $|s|-$dimensional cube will have the same distribution. The binomial coefficient takes care of the fact that we have $\binom{d}{j}$ many projections for a given dimension $1 \leq j \leq d$.
\end{proof}

\section*{Acknowledgements}
The authors would like to thank Markus Kiderlen for fruitful discussions and help with the simplification of \eqref{eq:doublesum}.

%%%%%%%%%%%%%%%%%%%
% References
%%%%%%%%%%%%%%%%%%%

\end{document}